\documentclass[11pt,reqno]{amsart}
\usepackage[letterpaper,hmargin=1.25in,vmargin=1.5in]{geometry}
\usepackage{hyperref}
\usepackage{amssymb,amsmath,amsfonts,latexsym}
\usepackage{bm}
\usepackage{mathtools}
\usepackage{enumerate}
\usepackage{enumitem}
\usepackage{mathrsfs}
\usepackage{array,graphics,color}
\usepackage{csquotes}
\allowdisplaybreaks

\numberwithin{equation}{section}
\newtheorem{dfn}{Definition}[section]
\newtheorem{thm}[dfn]{Theorem}
\newtheorem{lma}[dfn]{Lemma}

\newtheorem{ppsn}[dfn]{Proposition}
\newtheorem{crlre}[dfn]{Corollary}
\newtheorem{xmpl}[dfn]{Example}
\newtheorem{rmrk}[dfn]{Remark}

\DeclarePairedDelimiterX{\norm}[1]{\lVert}{\rVert}{#1}
\DeclarePairedDelimiterX{\bnorm}[1]{\big\lVert}{\big\rVert}{#1}
\DeclarePairedDelimiterX{\Bnorm}[1]{\Big\lVert}{\Big\rVert}{#1}

\begin{document}
	%\today
	
	\title[On some topology generated by $\mathcal{I}$-density function]{On some topology generated by $\mathcal{I}$-density function}

\author[Debnath] {Indrajit Debnath}
	\address{Department of Mathematics, The University of Burdwan, Burdwan-713104, West Bengal, India}
	\email{ind31math@gmail.com}
 
	\author[Banerjee] {Amar Kumar Banerjee}
	\address{Department of Mathematics, The University of Burdwan, Burdwan-713104, West Bengal, India}
	\email{akbanerjee1971@gmail.com, akbanerjee@math.buruniv.ac.in}

	
	
	\begin{abstract}
		In this paper we have studied on $\mathcal{I}$-density function using the notion of $\mathcal{I}$-density, introduced by Banerjee and Debnath \cite{banerjee 4} where $\mathcal{I}$ is an ideal of subsets of the set of natural numbers. We have explored certain properties of $\mathcal{I}$-density function and induced a topology using this function in the space of reals namely $\mathcal{I}$-density topology and we have given a characterization of the Lebesgue measurable subsets of reals in terms of Borel sets in $\mathcal{I}$-density topology. 
	\end{abstract}

	\maketitle
		\section{Introduction and Preliminaries}
 The convergence of sequences plays a significant role in the study of basic
mathematical theory. The idea of statistical convergence of sequences was introduced in the middle of twentieth century by H. Fast \cite{Fast}. For $K \subset \mathbb{N}$, the set of natural numbers and $n \in \mathbb{N}$ let $K_n =\{k\in K: k \leq n\}$. The natural density of the set $K$ is defined by $d(K)=\lim_{n \rightarrow{\infty}}\frac{|K_n|}{n}$, provided the limit exists, where $|K_n|$ stands for the cardinality of the set $K_n$. A sequence $\{x_n\}_{n \in \mathbb{N}}$ of real numbers is said to be statistically convergent to $x_0$ if for each $\epsilon > 0$ the set $K(\epsilon) = \{k\in \mathbb{N} : |x_k - x_0| \geq \epsilon\}$ has natural density zero.
% In the beginning of 1980s, this concept was developed by the published studies
% of Fridy [14] and Sal´at [28]. Later on statistical convergence was further studied ˇ
% from the sequence space point of view by [16], [17], [23] and also this concept as a
% summability method was studied by [4], [13], [26].

	% To include larger collection of sequences under purview the idea of convergence of real sequences was generalized to the notion of statistical convergence by H. Fast \cite{Fast} in the year 1951. 

	Later on, in the year $2000$, statistical convergence of real sequences were generalized to the idea of $\mathcal{I}$-convergence of real sequences by P. Kostyrko et al. \cite{Kostyrko 2000} using the notion of ideal  $\mathcal{I}$ of subsets of $\mathbb{N}$, the set of natural numbers. A subcollection  $\mathcal{I} \subset 2^\mathbb{N}$ is called an ideal if $A,B \in  \mathcal{I}$ implies $A \cup B \in  \mathcal{I}$ and $A\in  \mathcal{I}, B\subset A$ imply $B\in \mathcal{I}$. $\mathcal{I}$ is called nontrivial ideal if $\mathcal{I} \neq \{\emptyset\}$ and $\mathbb{N}\notin \mathcal{I}$. $\mathcal{I}$ is called admissible if it contains all the singletons. It is easy to verify that the family $\mathcal{I}_d=\{A \subset \mathbb{N}: d(A)=0\}$ forms a non-trivial admissible ideal of subsets of $\mathbb{N}$. If $\mathcal{I}$ is a proper non-trivial ideal then the family of sets $\mathcal{F}(\mathcal{I}) = \{M\subset \mathbb{N} : \mathbb{N}\setminus M \in \mathcal{I}\}$ is a filter on $\mathbb{N}$ and it is called the filter associated with the ideal $\mathcal{I}$ of $\mathbb{N}$.

	A sequence $\{x_n\}_{n \in \mathbb{N}}$ of real numbers is said to be $\mathcal{I}$-convergent \cite{Kostyrko 2000} to $x_0$ if the set $K(\epsilon) = \{k\in \mathbb{N} : |x_k - x_0| \geq \epsilon\}$ belongs to $\mathcal{I}$ for each $\epsilon>0$. A sequence $\{x_n\}_{n \in \mathbb{N}}$ of real numbers is said to be $\mathcal{I}$-bounded if there is a real number $M>0$ such that $\{k \in \mathbb{N}:|x_k|>M\}\in \mathcal{I}$. Further many works were carried out in this direction by many authors \cite{banerjee 2,banerjee 3,Lahiri 2005}.
	
	K. Demirci \cite{Demirci} introduced the notion of  $\mathcal{I}$-limit superior and inferior of real sequence and proved several basic properties.

	Let $\mathcal{I}$ be an admissible ideal in $\mathbb{N}$ and $x=\{x_n\}_{n \in \mathbb{N}}$ be a real sequence. Let,
	$B_x =\{b\in \mathbb{R} : \{k:x_k > b\}\notin \mathcal{I}\}$ and $A_x =\{a\in \mathbb{R} : \{k:x_k < a\}\notin \mathcal{I}\}$. Then the $\mathcal{I}$-limit superior of $x$ is given by,
	\begin{equation*}
		\mathcal{I}-\limsup  x  = \left\{
		\begin{array}{lr}
			\sup B_x & \text{if}\ B_x \neq \phi\\
			- \infty & \text{if}\ B_x = \phi
			
		\end{array} 
		\right.
	\end{equation*}

	and the $\mathcal{I}$-limit inferior of $x$ is given by,
	\begin{equation*}
		\mathcal{I}-\liminf  x = \left\{
		\begin{array}{lr}
			\inf A_x & \text{if}\ A_x \neq \phi\\
			\infty & \text{if}\ A_x = \phi
			
		\end{array} 
		\right.
	\end{equation*}
	Further Lahiri and Das \cite{Lahiri 2003} carried out some more works in this direction. Throughout the paper the ideal $\mathcal{I}$ will always stand for a nontrivial admissible ideal of subsets of $\mathbb{N}$.

	In 1961, Casper Goffman and Daniel Waterman \cite{Goffman} introduced the notion of homogeneity of sets relative to metric density and  Euclidean $n$-space was topologized by taking the homogeneous sets as open sets and this topology was referred to as $d$-topology or density topology.	The idea of density functions and the corresponding density topology were studied in several spaces like the space of real numbers \cite{Riesz}, Euclidean $n$-space \cite{Troyer}, metric spaces \cite{Lahiri 1998} etc.

	%For, $E\in \mathcal{L}$ and $x\in \mathbb{R}$ the upper density of $E$ at the point $x$ denoted by $d^-(x,E)$ and the lower density of $E$ at the point $x$ denoted by $d_-(x,E)$ are defined in \cite{White} as follows:
	%\begin{equation*}
	%	d^-(x,E)= \lim_{n \to \infty} \left( \sup \left\{\frac{m(E \cap I)}{m(I)}:I  \ \textrm{is a closed interval},  x\in I, 0<m(I)<\frac{1}{n}\right\} \right)
	%\end{equation*}
	
	%\begin{equation*}
		%d_-(x,E)= \lim_{n \to \infty} \left( \inf \left\{\frac{m(E \cap I)}{m(I)}:I \ \textrm{is a closed interval},  x\in I, 0<m(I)<\frac{1}{n}\right\} \right)
	%\end{equation*}
	%If $d_-(x,E)=d^-(x,E)=\gamma$ we say $E$ has density $\gamma$ at the point $x$ and denote $\gamma$ by $d(x,E)$. Moreover $x\in \mathbb{R}$ is a density point of $E$ if and only if $d(x,E)=1$. Let us take the family
	%\begin{center}
		%$\mathfrak{T}_d=\{E \in \mathcal{L}:d(x,E)=1 \  \mbox{for all} \ x \in E\}$
	%\end{center}
	%Then $\mathfrak{T}_d$ is ordinary density topology on $\mathbb{R}$ \cite{Goffman 1961} and it is finer than the usual topology $\mathfrak{T}_U$. Any member of $\mathfrak{T}_d$ is called $d$-open set.

	In the recent past the notion of classical Lebesgue density point was generalized by many authors by weakening the assumptions on the sequences of intervals and as a result several notions like $\langle s \rangle$-density point by M. Filipczak and J. Hejduk \cite{Filipczak 2004}, $\mathcal{J}$-density point by J. Hejduk and R. Wiertelak  \cite{Hejduk 2014}, $\mathcal{S}$-density point by F. Strobin and R. Wiertelak \cite{Strobin} were obtained. Significant generalizations on density topology was studied by W. Wojdowski in \cite{Wojdowski, Wojdowski 2013}. 
 
 % Recenty, Banerjee and Debnath \cite{banerjee 4} introduced the notion of $\mathcal{I}$-density introducing the notions of upper $\mathcal{I}$-density and lower $\mathcal{I}$-density.

	In this paper we try to give the notion of $\mathcal{I}$-density function with the help of $\mathcal{I}$-density introduced by Banerjee and Debnath \cite{banerjee 4} in the space of reals. In Section \ref{section 3} we have explored some properties of this function. We have considered $\mathcal{T}^{\mathcal{I}}$ to be the collection of measurable subsets of $\mathbb{R}$ such that each point of the set is an $\mathcal{I}$-density point. Finally in Section \ref{section 4} we have proved that the collection $\mathcal{T}^{\mathcal{I}}$ forms a topology on the set of reals. The mode of proofs are different from that given in \cite{banerjee 4}. We have characterized the Lebesgue measurable sets in the usual topology on reals as the Borel sets in $\mathcal{T}^{\mathcal{I}}$.

 We shall use the notation $\mathcal{L}$ for the $\sigma$-algebra of Lebesgue measurable sets and $\lambda$ for the Lebesgue measure. Throughout $\mathbb{R}$ stands for the set of all real numbers. The symbol $\mathcal{T}_U$ stands for the natural topology on $\mathbb{R}$. Wherever we write $\mathbb{R}$ it means that $\mathbb{R}$ is equipped with natural topology unless otherwise stated. The symmetric difference of two sets $A$ and $B$ is $(A \setminus B)\cup (B \setminus A)$ and it is denoted by $A \triangle B$. By \enquote{a sequence of closed intervals $\{J_n\}_{n \in \mathbb{N}}$ about a point $p$} we mean $p \in \bigcap_{n \in \mathbb{N}}J_n$.

	\section{$\mathcal{I}$-density}\label{section 2}

	\begin{dfn} [\cite{banerjee 4}]
		For a Lebesgue measurable set $E\in \mathcal{L}$, a point $p\in \mathbb{R}$ and $n \in \mathbb{N}$ the upper $\mathcal{I}$-density of $E$ at the point $p$ denoted by $\mathcal{I}-d^-(p,E)$ and the lower $\mathcal{I}$-density of $E$ at the point $p$ denoted by $\mathcal{I}-d_-(p,E)$ are defined as follows: 
		Suppose $\{J_n\}_{n \in \mathbb{N}}$ be a sequence of closed intervals about $p$ such that
		\begin{center}
			$\mathscr{S}(J_n)=\{n\in \mathbb{N}:0<\lambda(J_n)<\frac{1}{n}\} \in \mathcal{F}(\mathcal{I})$
		\end{center} 
	% Now we consider the following collection
	% \begin{center}
	% 	$\mathscr{C}_{p(\mathcal{I})}=\left\{\{J_n\}_{n \in \mathbb{N}}:\{J_n\}_{n \in \mathbb{N}} \mbox{be any sequence of closed intervals about} \ p \ \mbox{such that} \ \mathscr{S}(J_n) \in \mathcal{F}(\mathcal{I}) \right\}$
	% \end{center}
	For any such $\{J_n\}_{n \in \mathbb{N}}$ we take
		\begin{equation*}
			x_n = \frac{\lambda(J_n \cap E)}{\lambda(J_n)}  \ \mbox{for all} \ n \in \mathbb{N}
		\end{equation*}
		Then $\{x_n\}_{n \in \mathbb{N}}$ is a sequence of non-negative real numbers. Now if 
		$$B_{x_k} =\{b\in \mathbb{R} : \{k:x_k > b\}\notin \mathcal{I}\}$$ and $$A_{x_k} =\{a\in \mathbb{R} : \{k:x_k < a\}\notin \mathcal{I}\}$$ we define,
				\begin{align*}
				\mathcal{I}-d^-(p,E) &= \sup \{ \sup B_{x_n}: \{J_n\}_{n \in \mathbb{N}} \ \mbox{such that} \ \mathscr{S}(J_n) \in \mathcal{F}(\mathcal{I})\} \\
				&= \sup \{\mathcal{I}-\limsup \ x_n: \{J_n\}_{n \in \mathbb{N}} \ \mbox{such that} \ \mathscr{S}(J_n) \in \mathcal{F}(\mathcal{I})\} 
			\end{align*}
		and 
		\begin{align*}
			\mathcal{I}-d_{-}(p,E) &= \inf \{ \inf A_{x_n}: \{J_n\}_{n \in \mathbb{N}} \ \mbox{such that} \ \mathscr{S}(J_n) \in \mathcal{F}(\mathcal{I})\} \\
			&= \inf \{\mathcal{I}-\liminf \ x_n: \{J_n\}_{n \in \mathbb{N}} \ \mbox{such that} \ \mathscr{S}(J_n) \in \mathcal{F}(\mathcal{I})\} 
		\end{align*}

  In the above two expressions it is to be understood that $\{J_n\}_{n \in \mathbb{N}}$'s are closed intervals about the point $p$. Now, if $\mathcal{I}-d_-(p,E)=\mathcal{I}-d^-(p,E)$ then we denote the common value by $\mathcal{I}-d(p,E)$ which we call as $\mathcal{I}$-density of $E$ at the point $p$.
	\end{dfn}
	
	A point $p_0 \in \mathbb{R}$ is an $\mathcal{I}$-\textit{density point} of $E\in \mathcal{L}$ if $\mathcal{I}-d(p_0,E)=1$. 
	
	If a point $p_0\in \mathbb{R}$ is an $\mathcal{I}$-density point of the set $\mathbb{R}\setminus E$, then $p_0$ is an $\mathcal{I}$-\textit{dispersion point} of $E$.

		\begin{rmrk}
		The notion of $\mathcal{I}$-density point is more general than the notion of density point as the collection of intervals about the point $p$ considered in case of  $\mathcal{I}$-density is larger than that considered in case of classical density which is illustrated in the following example.
	\end{rmrk}
	
	\begin{xmpl}
		Let us consider the ideal $\mathcal{I}_d$ of subsets of $\mathbb{N}$ where $\mathcal{I}_d$ is the ideal containing all those subsets of $\mathbb{N}$ whose natural density is zero. Now, for any point $p \in \mathbb{R}$ consider the following collections of sequences of intervals:
		\begin{center}
			$\mathscr{C}_{p(\mathcal{I}_{fin})}=\{\{J_n\}_{n \in \mathbb{N}}: \{J_n\} \textrm{is a sequence of closed intervals about} \ p \ \textrm{such that} \  \mathscr{S}(J_n) \in \mathcal{F}(\mathcal{I}_{fin}) \}$ and 
		\end{center}
		\begin{center}
			$\mathscr{C}_{p(\mathcal{I}_d)}=\{\{J_n\}_{n \in \mathbb{N}}: \{J_n\} \textrm{is a sequence of closed intervals about} \ p \ \textrm{such that} \  \mathscr{S}(J_n)\in \mathcal{F}(\mathcal{I}_d) \}$
		\end{center}
		We claim that $\mathscr{C}_{p(\mathcal{I}_{fin})} \subsetneqq \mathscr{C}_{p(\mathcal{I}_d)}$. Since any finite subset of $\mathbb{N}$ has natural density zero so $\mathcal{I}_{fin} \subset \mathcal{I}_{d}$.\\ 
		Now in particular let us take the following sequence $\{K_n\}_{n \in \mathbb{N}}$ of closed intervals about a point $p$.

	\begin{equation*}
		K_n =
		\left\{
		\begin{array}{ll}
			\left[p-\frac{1}{2n+1},p+\frac{1}{2n+1}\right]  & \textrm{for} \ n \neq m^2 \ \textrm{where} \ m \in \mathbb{N} \\
			\left[p-n,p+n\right] & \textrm{for} \ n=m^2 \ \textrm{where} \ m \in \mathbb{N}
		\end{array}
		\right.
	\end{equation*}
		
		We observe that for $n \neq m^2$, $\lambda(K_n)= \frac{2}{2n+1} < \frac{1}{n}$ and for $n=m^2$, $\lambda(K_{n})=2n \nless \frac{1}{n}$. Therefore, $\mathscr{S}(K_n)=\{n \in \mathbb{N}: 0< \lambda(K_n)<\frac{1}{n}\}=\{n: n \neq m^2\} \in \mathcal{F}(\mathcal{I}_d)$. But since $\mathbb{N} \setminus \mathscr{S}(K_n)=\{n: n = m^2\}$ is not a finite set so it doesnot belong to $\mathcal{I}_{fin}$. As a result $\{K_n\}\in \mathscr{C}_{p(\mathcal{I}_d)} \setminus \mathscr{C}_{p(\mathcal{I}_{fin})}$\\
		
		Let us take the set $E$ to be the open interval $(-1,1)$ and the point $p$ to be $0$. Let $\{K_n\}_{n \in \mathbb{N}} \in \mathscr{C}_{0(\mathcal{I}_d)} \setminus \mathscr{C}_{0(\mathcal{I}_{fin})}$ be taken as above. Now if $x_n=\frac{\lambda(K_n \cap E)}{\lambda(K_n)}$ then
		
		\begin{equation*}
			x_n =
			\left\{
			\begin{array}{ll}
				1  & \mbox{if } n \neq m^2\\
				\frac{1}{m^2} & \mbox{if } n=m^2
			\end{array}
			\right.
		\end{equation*}
	Now let us calculate $\limsup$ and $\liminf$ of the sequence $\{x_n\}$. Thus,
 $$ \limsup x_n = \inf_{n} \sup_{k \geq n} x_k =1 \ \mbox{and} \ \liminf x_n = \sup_{n} \inf_{k \geq n} x_k = 0$$
 
 Consequently, $\lim_n x_n$ does not exists. 
	\\
	Now let us calculate $\mathcal{I}_d-\limsup$ and $\mathcal{I}_d-\liminf$ of the sequence $\{x_n\}$. Here $$B_{x_n} =\{b\in \mathbb{R} : \{k:x_k > b\}\notin \mathcal{I}_d\}=(-\infty,1)$$ and $$A_{x_n} =\{a\in \mathbb{R} : \{k:x_k < a\}\notin \mathcal{I}_d\}=(1,\infty)$$
	So, $\mathcal{I}_d-\limsup x_n = \sup B_{x_n}=1$ and $\mathcal{I}_d-\liminf x_n = \inf A_{x_n}=1$. Thus both $\mathcal{I}_d-\limsup$ and $\mathcal{I}_d-\liminf$ of the sequence $\{x_n\}$ are equal and equals to 1. Next we will show that $0$ is $\mathcal{I}_d$-density point of the set $E$. 
	
	Given any sequence of closed intervals $\{J_n\}_{n \in \mathbb{N}}$ about the point $0$ such that $\mathscr{S}(J_n)\in \mathcal{F}(\mathcal{I}_d)$ we have $\{n:J_n \subset E\} \in \mathcal{F}(\mathcal{I}_d)$. For if $\mathscr{S}(J_n)=\{k_1<k_2<\cdots<k_n<\cdots\}$ (say). Then there exists $n_0 \in \mathbb{N}$ such that for $k_n > k_{n_0}$, $J_{k_n} \subset E$. Thus, $\{n:J_n \subset E\} \supset \mathscr{S}(J_n) \setminus \{k_1,k_2,\cdots,k_{n_0}\}$. Since $\mathbb{N} \setminus \{k_1,k_2,\cdots,k_{n_0}\} \in \mathcal{F}(\mathcal{I}_d)$ so $$\mathscr{S}(J_n) \setminus \{k_1,k_2,\cdots,k_{n_0}\}=\mathscr{S}(J_n) \cap (\mathbb{N} \setminus \{k_1,k_2,\cdots,k_{n_0}\}) \in \mathcal{F}(\mathcal{I}_d)$$ Now if, $J_n \subset E$ then $r_n=\frac{\lambda(J_n \cap E)}{\lambda(J_n)}=\frac{\lambda(J_n)}{\lambda(J_n)}=1$. Thus, $\{n:r_n=1\} \supset \{n:J_n \subset E\}$. Therefore, $\{n:r_n=1\} \in \mathcal{F}(\mathcal{I}_d)$. Therefore, $B_{r_n}=(-\infty,1)$ and $A_{r_n}=(1,\infty)$ and so, $\mathcal{I}_d-\limsup r_n = \sup B_{r_n}=1$ and $\mathcal{I}_d-\liminf r_n = \inf A_{r_n}=1$. This is true for all $\{J_n\}_{n \in \mathbb{N}} \in \mathscr{C}_{0(\mathcal{I}_d)}$. Hence,
	
		\begin{align*}
		\mathcal{I}_d-d^-(0,E) &= \sup \{ \sup B_{r_n}: \{J_n\}_{n \in \mathbb{N}} \ \mbox{such that} \ \mathscr{S}(J_n) \in \mathcal{F}(\mathcal{I}_d)\} =1
	\end{align*}
	and
	\begin{align*}
		\mathcal{I}_d-d_{-}(0,E) &= \inf \{ \inf A_{r_n}: \{J_n\}_{n \in \mathbb{N}} \ \mbox{such that} \ \mathscr{S}(J_n) \in \mathcal{F}(\mathcal{I}_d)\} =1
	\end{align*}
		Hence $\mathcal{I}_d-d(0,E)$ exists and equals to $1$. So, 0 is an $\mathcal{I}_d$-density point of the set $E$.
	\end{xmpl}

 Here we are stating some important results which will be needed later in our discussion and for the sake of completeness we are giving the proof of the results cited from \cite{banerjee 4}.
	
	\begin{thm}[\cite{Demirci}]{For any real sequence $x$, $\mathcal{I}-\liminf x \leq \mathcal{I}-\limsup x$.}
	\end{thm}
	
	\begin{thm}[\cite{banerjee 4}] For any Lebesgue measurable set $A \subset \mathbb{R}$ and any point $p \in \mathbb{R}$,
		\begin{center}
			$\mathcal{I}-d_-(p,A) \leq \mathcal{I}-d^{-}(p,A)$
		\end{center}
	\end{thm}

  \begin{thm}\label{1}[\cite{Lahiri 2003}]If $x=\{x_n\}_{n \in \mathbb{N}}$ and $y=\{y_n\}_{n \in \mathbb{N}}$ are two $\mathcal{I}$-bounded real number sequences, then
	\end{thm}
	
	\begin{enumerate}[label=(\roman*)]
		\item $\mathcal{I}-\limsup (x+y) \leq \mathcal{I}-\limsup x +\mathcal{I}-\limsup y$
		\item $\mathcal{I}-\liminf (x+y) \geq \mathcal{I}-\liminf x + \mathcal{I}-\liminf y$
	\end{enumerate}

	\begin{ppsn}\label{2}[\cite{banerjee 4}] Given an $\mathcal{I}$-bounded real sequence $\{x_n\}_{n \in \mathbb{N}}$ and a real number $c$, 
	\end{ppsn}
	\begin{center}
		\begin{enumerate}[label=(\roman*)]
			\item $\mathcal{I}-\liminf (c+x_n)=c+\mathcal{I}-\liminf x_n$
			\item $\mathcal{I}-\limsup (c+x_n)=c+\mathcal{I}-\limsup x_n$
		\end{enumerate}
	\end{center}

	\begin{ppsn}[\cite{banerjee 4}]
		{For any real sequence $x=\{x_n\}_{n \in \mathbb{N}}$,
		}
	\end{ppsn}
	\begin{enumerate}[label=(\roman*)]
		\item $\mathcal{I}-\limsup (-x)=-(\mathcal{I}-\liminf x)$
		\item $\mathcal{I}-\liminf (-x)=-(\mathcal{I}-\limsup x)$
	\end{enumerate}

	\begin{lma}\label{3}[\cite{banerjee 4}]
		For any disjoint Lebesgue measurable subsets $A$ and $B$ of $\mathbb{R}$ and any point $p \in \mathbb{R}$ if $\mathcal{I}-d (p,A)$ and $\mathcal{I}-d (p,B)$ exist, then $\mathcal{I}-d (p,A \cup B)$ exists and $\mathcal{I}-d (p,A \cup B)=\mathcal{I}-d (p,A)+\mathcal{I}-d (p,B)$. 
	\end{lma}

	\begin{lma}[\cite{banerjee 4}]
		If $\mathcal{I}-d(p,A)$ and $\mathcal{I}-d(p,B)$ exist and $A \subset B$. Then $\mathcal{I}-d(p,B \setminus A)$ exists and $\mathcal{I}-d(p,B \setminus A)=\mathcal{I}-d(p,B)-\mathcal{I}-d(p,A)$
	\end{lma}
	
	\begin{thm}[\cite{banerjee 4}]
		For any measurable set $H$, $\mathcal{I}$-density of $H$ at a point $p$ exists if and only if $\mathcal{I}-d^{-}(p,H)+\mathcal{I}-d^{-}(p,H^{c})=1$.
	\end{thm}

	Let $H \subset \mathbb{R}$ be a measurable set. Let us denote the set of points of $\mathbb{R}$ at which $H$ has lower $\mathcal{I}$-density 1 by $\Theta_\mathcal{I}(H)$. i.e.
 \begin{equation*}
   \Theta_\mathcal{I}(H)=\{x \in \mathbb{R}: \mathcal{I}-d_{-}(x,H)=1 \}
\end{equation*}

Next we state the Lebesgue $\mathcal{I}$-density theorem which is as follows.

	\begin{thm}\label{4}[\cite{banerjee 4}]
		For any measurable set $H \subset \mathbb{R},$
		\begin{center}
			$\lambda(H \triangle \Theta_\mathcal{I}(H))=0$ 
		\end{center}
		
	\end{thm}

	The statement of this theorem may also be read as almost all points of an arbitrary measurable set $H$ are points of $\mathcal{I}$-density for $H$ and moreover we can conclude that $\Theta_{\mathcal{I}}(H)$ is measurable.
	
	\section{$\mathcal{I}$-density function}\label{section 3}

The function $\Theta_\mathcal{I}(.): \mathcal{L} \rightarrow 2^\mathbb{R}$ is called $\mathcal{I}$-density function since $\Theta_\mathcal{I}$ takes measurable set as input and returns the set of all points that have $\mathcal{I}$-density $1$ in $H$. Now we explore some properties of $\mathcal{I}$-density function.

\begin{ppsn}\label{5}
If $A$ and $B$ are measurable sets and $\lambda(A \triangle B)=0$ then $\Theta_\mathcal{I}(A)=\Theta_\mathcal{I}(B)$. 
\end{ppsn}

\begin{proof}
Let $\{J_n\}_{n \in \mathbb{N}}$ be any sequence of closed interval in $\mathbb{R}$. If $\lambda(A \triangle B)=0$ then we claim that $\lambda(A \cap J_n)=\lambda(B \cap J_n)$ for each interval $J_n \subset \mathbb{R}$. Now
      \begin{equation*}
          \begin{split}
              A &=A \cap (B \cup B^{c})\\
              &= (A \cap B) \cup (A \cap B^{c})\\
              &= (A \cap B) \cup (A \setminus B)\\
              &\subset B \cup (A \triangle B)
          \end{split}
      \end{equation*}
      
      For any $J_n \subset \mathbb{R}$ we have
      \begin{equation*}
          \begin{split}
              \lambda(A \cap J_n) & \leq \lambda ((B \cup (A \triangle B)) \cap J_n)\\
              & \leq \lambda ((A \triangle B)) \cap J_n)+ \lambda (B \cap J_n)\\
              &= \lambda(B \cap J_n) \quad \mbox{since} \ \lambda ((A \triangle B)) \cap J_n) \leq \lambda(A \triangle B)=0 
          \end{split}
      \end{equation*}
      Similarly, $\lambda (B \cap J_n) \leq \lambda(A \cap J_n)$ for any interval $J_n \subset \mathbb{R}$. So we have $\lambda(A \cap J_n)=\lambda(B \cap J_n)$ for any interval $J_n \subset \mathbb{R}$. Now we are to show $\Theta_\mathcal{I}(A)=\Theta_\mathcal{I}(B)$. 

      Let $x \in \Theta_\mathcal{I}(A)$. So, $\mathcal{I}-d_{-}(x,A)=1$. Now,
		\begin{align*}
			\mathcal{I}-d_-(x,A) &=\inf \{\mathcal{I}-\liminf \frac{\lambda(A \cap J_n)}{\lambda(J_n)} : \{J_n\}_{n \in \mathbb{N}} \ \mbox{such that} \ \mathscr{S}(J_n) \in \mathcal{F}(\mathcal{I})\}\\ 
			&= \inf \{\mathcal{I}-\liminf \frac{\lambda(B \cap J_n)}{\lambda(J_n)} : \{J_n\}_{n \in \mathbb{N}} \ \mbox{such that} \ \mathscr{S}(J_n) \in \mathcal{F}(\mathcal{I})\} \\
			&= \mathcal{I}-d_-(x,B)
		\end{align*}
      
      So, $\mathcal{I}-d_-(x,B)=1$ and hence $x \in \Theta_\mathcal{I}(B)$. So, $\Theta_\mathcal{I}(A) \subseteq \Theta_\mathcal{I}(B)$. Similarly, $\Theta_\mathcal{I}(B) \subseteq \Theta_\mathcal{I}(A)$. Thus, $\Theta_\mathcal{I}(A)=\Theta_\mathcal{I}(B)$. This completes the proof.
      
\end{proof}

\begin{crlre}
    Let $A \subseteq \mathbb{R}$ be measurable then $\Theta_{\mathcal{I}}(A)=\Theta_{\mathcal{I}}(\Theta_{\mathcal{I}}(A))$ i.e. $\mathcal{I}$-density function is idempotent.
\end{crlre}
\begin{proof}
    By Lebesgue $\mathcal{I}$-Density Theorem \ref{4} $\lambda (A \triangle \Theta_\mathcal{I}(A))=0$. So by Proposition \ref{5} we have $\Theta_{\mathcal{I}}(A)=\Theta_{\mathcal{I}}(\Theta_{\mathcal{I}}(A))$.

\end{proof}

\begin{lma}\label{6}
  Given a pair of Lebesgue measurable sets $A$ and $B$ such that $A \subseteq B$, $\Theta_{\mathcal{I}}(A) \subseteq \Theta_{\mathcal{I}}(B)$ i.e. $\mathcal{I}$-density function is monotonic.  
\end{lma}
    
\begin{proof}
    If $A \subseteq B$, $\lambda(A \cap J_n) \leq \lambda(B \cap J_n)$ for each interval $J_n \subset \mathbb{R}$. So if $x \in \Theta_\mathcal{I}(A)$ then $\mathcal{I}-d_-(x,A)=1$. Hence,
    \begin{align*}
			\mathcal{I}-d_-(x,A) &=\inf \{\mathcal{I}-\liminf \frac{\lambda(A \cap J_n)}{\lambda(J_n)} : \{J_n\}_{n \in \mathbb{N}} \ \mbox{such that} \ \mathscr{S}(J_n) \in \mathcal{F}(\mathcal{I})\}\\ 
			&\leq \inf \{\mathcal{I}-\liminf \frac{\lambda(B \cap J_n)}{\lambda(J_n)} : \{J_n\}_{n \in \mathbb{N}} \ \mbox{such that} \ \mathscr{S}(J_n) \in \mathcal{F}(\mathcal{I})\} \\
			&= \mathcal{I}-d_-(x,B)
		\end{align*}
   Hence, $\mathcal{I}-d_-(x,B)\geq 1$. So, $\mathcal{I}-d_-(x,B)=1$ and hence $x \in \Theta_\mathcal{I}(B)$. Consequently, $\Theta_{\mathcal{I}}(A) \subseteq \Theta_{\mathcal{I}}(B)$.
\end{proof}

\begin{thm}\label{7}
    For every pair of Lebesgue measurable sets $A,B \in \mathcal{L}$, $\Theta_{\mathcal{I}}(A \cap B)=\Theta_{\mathcal{I}}(A) \cap \Theta_{\mathcal{I}}(B)$.
\end{thm}
\begin{proof} Since $A \cap B \subseteq A$ and $A \cap B \subseteq B$, so by Lemma \ref{6}, $\Theta_{\mathcal{I}}(A \cap B) \subseteq \Theta_{\mathcal{I}}(A)$ and $\Theta_{\mathcal{I}}(A \cap B) \subseteq \Theta_{\mathcal{I}}(B)$. Consequently, $\Theta_{\mathcal{I}}(A \cap B) \subseteq \Theta_{\mathcal{I}}(A) \cap \Theta_{\mathcal{I}}(B)$. Now we are to prove $ \Theta_{\mathcal{I}}(A) \cap \Theta_{\mathcal{I}}(B) \subseteq \Theta_{\mathcal{I}}(A \cap B)$. Let $x \in \Theta_{\mathcal{I}}(A) \cap \Theta_{\mathcal{I}}(B)$. Thus $x \in \Theta_{\mathcal{I}}(A)$ and $x \in  \Theta_{\mathcal{I}}(B)$. So, $\mathcal{I}-d_{-}(x,A)=1$ and $\mathcal{I}-d_{-}(x,B)=1$. We are to show $\mathcal{I}-d_{-}(x,A \cap B)=1$. It is sufficient to show $\mathcal{I}-d_{-}(x,A \cap B)\geq1$. 

Let $\{I_k\}_{k \in \mathbb{N}}$ be any sequence of closed intervals about the point $x$ such that $\mathscr{S}(I_k) \in \mathcal{F}(\mathcal{I})$. Then for all $k \in \mathscr{S}(I_k)$, $\lambda(A \cap I_k)+\lambda(B \cap I_k)-\lambda(A \cap B \cap I_k) \leq \lambda(I_k)$.

So, for $k \in \mathscr{S}(I_k)$ we have 
		\begin{equation*}
			\frac{\lambda(A \cap I_k)}{\lambda(I_k)}+\frac{\lambda(B \cap I_k)}{\lambda(I_k)} \leq 1+ \frac{\lambda((A \cap B) \cap I_k)}{\lambda(I_k)}
		\end{equation*}
	
	Let us take $x_k=\frac{\lambda(A \cap I_k)}{\lambda(I_k)}, y_k=\frac{\lambda(B \cap I_k)}{\lambda(I_k)}, z_k=\frac{\lambda((A \cap B) \cap I_k)}{\lambda(I_k)}$. So, $z_k \geq x_k+y_k-1$. Thus,
\begin{align*}
     \mathcal{I}-\liminf z_n & \geq \mathcal{I}-\liminf (x_n+y_n-1)\\
	& = \mathcal{I}-\liminf (x_n+y_n) - 1 \ \mbox{by Proposition \ref{2}}\\ 
		& \geq \mathcal{I}-\liminf x_n+ \mathcal{I}-\liminf y_n - 1 \ \mbox{by Theorem \ref{1}}
 \end{align*}

 Hence,
 \begin{multline*}
      \inf \{\mathcal{I}-\liminf z_n: \{I_n\} \ \mbox{such that} \ \mathscr{S}(I_n) \in \mathcal{F}(\mathcal{I}) \}\\
      \geq  \inf \{\mathcal{I}-\liminf x_n + \mathcal{I}-\liminf y_n -1: \{I_n\} \ \mbox{such that} \ \mathscr{S}(I_n) \in \mathcal{F}(\mathcal{I}) \}\\
      \geq \inf \{\mathcal{I}-\liminf x_n: \{I_n\} \ \mbox{such that} \ \mathscr{S}(I_n) \in \mathcal{F}(\mathcal{I}) \}  \\
		     \qquad      \qquad    + \inf \{\mathcal{I}-\liminf y_n: \{I_n\} \ \mbox{such that} \ \mathscr{S}(I_n) \in \mathcal{F}(\mathcal{I}) \} -1
 \end{multline*}

 So,
 \begin{align*}
     \mathcal{I}-d_{-}(x,A\cap B) & =\inf \{\mathcal{I}-\liminf z_n: \{I_n\} \ \mbox{such that} \ \mathscr{S}(I_n) \in \mathcal{F}(\mathcal{I}) \}\\
     & \geq \mathcal{I}-d_{-}(x,A)+\mathcal{I}-d_{-}(x,B)-1 \\
     & = 1+1-1=1
 \end{align*}

	% \begin{align*}
	% 	\mathcal{I}-\liminf z_n &\geq \mathcal{I}-\liminf (x_n+y_n-1)\\
	% 	& \geq \mathcal{I}-\liminf (x_n+y_n) - 1 \\
	% 	& \geq \mathcal{I}-\liminf x_n+ \mathcal{I}-\liminf y_n - 1\\
	% 	& \geq \inf \{\mathcal{I}-\liminf x_n: \{I_n\} \ \mbox{such that} \ \mathscr{S}(I_n) \in \mathcal{F}(\mathcal{I}) \}  \\
	% 	&     \qquad      \qquad    + \inf \{\mathcal{I}-\liminf y_n: \{I_n\} \ \mbox{such that} \ \mathscr{S}(I_n) \in \mathcal{F}(\mathcal{I}) \} -1\\
	% 	& = \mathcal{I}-d_{-}(x,A)+\mathcal{I}-d_{-}(x,B)-1
	% \end{align*}
	% So, $$\mathcal{I}-d_{-}(x,A\cap B)=\inf \{\mathcal{I}-\liminf z_n: \{I_n\} \ \mbox{such that} \ \mathscr{S}(I_n) \in \mathcal{F}(\mathcal{I}) \} \geq 1+1-1=1$$
 Therefore, $\mathcal{I}-d_{-}(x,A\cap B)=1$. So, $x \in \Theta_{\mathcal{I}}(A \cap B)$. Hence, $ \Theta_{\mathcal{I}}(A) \cap \Theta_{\mathcal{I}}(B) \subseteq \Theta_{\mathcal{I}}(A \cap B)$. This completes the proof.
\end{proof}

\begin{lma}\label{8} Let $A,B \subseteq \mathbb{R}$ such that $\lambda(A \setminus B)=0$ then $\Theta_{\mathcal{I}}(A) \subseteq \Theta_{\mathcal{I}}(B)$.
\end{lma}
\begin{proof} Let us assume $\lambda(A \setminus B)=0$ and $\Theta_{\mathcal{I}}(A) \nsubseteq \Theta_{\mathcal{I}}(B)$. Then there exists $x \in \Theta_{\mathcal{I}}(A)$ such that $x \notin \Theta_{\mathcal{I}}(B)$ i.e $\mathcal{I}-d_{-}(x,A)=1$ but $\mathcal{I}-d_{-}(x,B)<1$. Now we have the following two cases.

Case(i) If $\mathcal{I}-d_{-}(x,A \setminus B)>0$ then 
$$\inf \{\mathcal{I}-\liminf \frac{\lambda((A \setminus B) \cap J_n)}{\lambda(J_n)} : \{J_n\}_{n \in \mathbb{N}} \ \mbox{such that} \ \mathscr{S}(J_n) \in \mathcal{F}(\mathcal{I})\}>0$$ 
So, for some sequence of closed intervals about $x$, say $\{J_n\}_{n \in \mathbb{N}} \ \mbox{such that} \ \mathscr{S}(J_n) \in \mathcal{F}(\mathcal{I})$ we have 
$$\mathcal{I}-\liminf \frac{\lambda((A \setminus B) \cap J_n)}{\lambda(J_n)} >0 $$
Thus for some $n_0 \in \mathbb{N}$ we have $\lambda((A \setminus B) \cap J_{n_0})>0$. So there exists a measurable subset $H \subseteq A$ such that $H \cap B = \emptyset$ and $\lambda(H) >0$. So, $\lambda(A \setminus B) \geq \lambda(H) >0$. It contradicts the fact that $\lambda(A \setminus B)=0$.

Case(ii): If $\mathcal{I}-d_{-}(x,A \setminus B)=0$. Note that $A$ can be written as a disjoint union of $(A \setminus B)$ and $(A \cap B)$. Since $(A \setminus B)$ and $(A \cap B)$ are measurable, so by Lemma \ref{3} 
$$\mathcal{I}-d_{-}(x,A)=\mathcal{I}-d_{-}(x,A \setminus B)+\mathcal{I}-d_{-}(x,A \cap B)=\mathcal{I}-d_{-}(x,A \cap B)$$ 
Thus $\mathcal{I}-d_{-}(x,A \cap B)=1$. Now since $A \cap B \subset B$ so $\mathcal{I}-d_{-}(x,A \cap B) \leq \mathcal{I}-d_{-}(x,B)$ which implies $\mathcal{I}-d_{-}(x,B) \geq 1$ which is a contradiction since $x \notin \Theta_{\mathcal{I}}(B)$.

    So our assumption was wrong. Hence the result follows.
\end{proof}

\begin{lma}\label{9}
    Let $A$ be any subset of $\mathbb{R}$ such that $\lambda(A)=0$ then $\Theta_{\mathcal{I}}(A)=\emptyset$ and $\Theta_{\mathcal{I}}(\mathbb{R} \setminus A)=\mathbb{R}$
\end{lma}

\begin{proof} If $\lambda(A)=0$ then at each point $x \in \mathbb{R}$ we have
\begin{align*}
			\mathcal{I}-d^-(x, A) &=\sup \{\mathcal{I}-\limsup \frac{\lambda( A \cap I_n)}{\lambda(I_n)} : \{I_n\}_{n \in \mathbb{N}} \ \mbox{such that} \ \mathscr{S}(I_n) \in \mathcal{F}(\mathcal{I})\}\\ 
			&= 0
		\end{align*}
  This implies, $\mathcal{I}-d_-(x, A)=0$. Thus, $\Theta_{\mathcal{I}}(A)$ is an empty set. 

    Clearly, $\Theta_{\mathcal{I}}(\mathbb{R} \setminus A) \subseteq \mathbb{R}$. We are to show $\mathbb{R} \subseteq \Theta_{\mathcal{I}}(\mathbb{R} \setminus A)$. Let $x \in \mathbb{R}$ and let $\{I_k\}_{k \in \mathbb{N}}$ be any sequence of closed intervals about $x$ such that $\mathscr{S}(I_k) \in \mathcal{F}(\mathcal{I})$. Then for $k \in \mathscr{S}(I_k)$ we have 
    $$\lambda(\mathbb{R} \cap I_k)=\lambda((\mathbb{R} \setminus A)\cap I_k)+\lambda(A \cap I_k)=\lambda((\mathbb{R} \setminus A)\cap I_k)$$
    Now,
    \begin{align*}
			\mathcal{I}-d_-(x,\mathbb{R} \setminus A) &=\inf \{\mathcal{I}-\liminf \frac{\lambda((\mathbb{R} \setminus A) \cap I_n)}{\lambda(I_n)} : \{I_n\}_{n \in \mathbb{N}} \ \mbox{such that} \ \mathscr{S}(I_n) \in \mathcal{F}(\mathcal{I})\}\\ 
			&= \inf \{\mathcal{I}-\liminf \frac{\lambda(\mathbb{R} \cap I_n)}{\lambda(I_n)} : \{I_n\}_{n \in \mathbb{N}} \ \mbox{such that} \ \mathscr{S}(I_n) \in \mathcal{F}(\mathcal{I})\} \\
			&= 1
		\end{align*}
  Thus, $x \in \Theta_{\mathcal{I}}(\mathbb{R} \setminus A)$. So, $\mathbb{R} \subseteq \Theta_{\mathcal{I}}(\mathbb{R} \setminus A)$.
\end{proof}

\begin{thm}
    If $A$ is a measurable subset of $\mathbb{R}$, then $\Theta_{\mathcal{I}}(A) \cap \Theta_{\mathcal{I}}(A^c)=\emptyset$
\end{thm}
\begin{proof} If possible let $\Theta_{\mathcal{I}}(A) \cap \Theta_{\mathcal{I}}(A^c)\neq \emptyset$. Then there exists a point $x \in \Theta_{\mathcal{I}}(A) \cap \Theta_{\mathcal{I}}(A^c)$. So, $\mathcal{I}-d_{-}(x,A)=1$ and $\mathcal{I}-d_{-}(x,A^{c})=1$. Thus,
$$\inf \{\mathcal{I}-\liminf \frac{\lambda(A \cap J_n)}{\lambda(J_n)} : \{J_n\}_{n \in \mathbb{N}} \ \mbox{such that} \ \mathscr{S}(J_n) \in \mathcal{F}(\mathcal{I})\}=1$$ and $$\inf \{\mathcal{I}-\liminf \frac{\lambda(A^{c} \cap J_n)}{\lambda(J_n)} : \{J_n\}_{n \in \mathbb{N}} \ \mbox{such that} \ \mathscr{S}(J_n) \in \mathcal{F}(\mathcal{I})\}=1$$
So for any fixed interval $\{J_n\}_{n \in \mathbb{N}}$, $\mathcal{I}-\liminf \frac{\lambda(A \cap J_n)}{\lambda(J_n)} \geq 1$ and $\mathcal{I}-\liminf \frac{\lambda(A^{c} \cap J_n)}{\lambda(J_n)} \geq 1$. Therefore,
$$\mathcal{I}-\liminf \frac{\lambda(A \cap J_n)}{\lambda(J_n)}+\mathcal{I}-\liminf \frac{\lambda(A^{c} \cap J_n)}{\lambda(J_n)} \geq 2$$ As a result,
$$\mathcal{I}-\liminf \left\{\frac{\lambda(A \cap J_n)}{\lambda(J_n)}+\frac{\lambda(A^{c} \cap J_n)}{\lambda(J_n)}\right\}=\mathcal{I}-\liminf \frac{\lambda(\mathbb{R} \cap J_n)}{\lambda(J_n)}=1 \geq 2$$ which is a contradiction. So the result follows.
    
\end{proof}

% \begin{rmrk}
%     For any arbitrary collection of measurable sets $\{A_\alpha\}_{\alpha \in \Lambda}$ where $\Lambda$ is arbitrary indexing set, whether the equality $\Theta_{\mathcal{I}}(\bigcup_{\alpha \in \Lambda}A_{\alpha})=\bigcup_{\alpha \in \Lambda}\Theta_{\mathcal{I}}(A_{\alpha})$ holds or not is our next question. Let us consider the following example.
   
% \end{rmrk}
 % \begin{enumerate}
 %        \item $\Theta_{\mathcal{I}}((0,1))\cup \Theta_{\mathcal{I}}((1,2))=(0,1)\cup (1,2) \neq (0,2) = \Theta_{\mathcal{I}}((0,1)\cup (1,2))$
 %        \item Let $A_{\alpha}=\mathbb{Q}+\alpha$ where $\mathbb{Q}$ is the set of rationals and $\alpha \in [0,1]$. Then $\lambda(A_\alpha)=0$ for all $\alpha$. So, by Lemma $4.7$ $\Theta_{\mathcal{I}}(A_\alpha)=\emptyset$. But $\bigcup_{\alpha \in [0,1]}A_{\alpha}=\mathbb{R}$. Thus, $$\Theta_{\mathcal{I}}(\bigcup_{\alpha \in [0,1]}A_{\alpha})= \Theta_{\mathcal{I}}(\mathbb{R})=\mathbb{R} \neq \emptyset = \bigcup_{\alpha \in [0,1]}\Theta_{\mathcal{I}}(A_\alpha)$$
 %    \end{enumerate}

 %    Thus, $\mathcal{I}$-density function fail to distribute over finite or arbitrary union.

	\section{$\mathcal{I}$-density topology}\label{section 4}

We consider $\mathcal{T}^{\mathcal{I}}$ to be the collection of measurable subsets of $\mathbb{R}$ such that each point of the set is $\mathcal{I}$-density point. So,
\begin{equation*}
    \mathcal{T}^{\mathcal{I}}=\{A \in \mathcal{L}: A \subset \Theta_{\mathcal{I}}(A)\}
\end{equation*}

Whether such a collection forms a topology is the next question. The difficulty lies in the fact that a topology must be closed under arbitrary unions and arbitrary union of measurable sets may not be necessarily measurable. 

\begin{thm}\label{10}
    The Lebesgue measure $\lambda$ on $\mathbb{R}$ satisfies the countable chain condition. i.e. any collection of measurable sets each with positive measure such that the intersection of two distinct elements of that collection has measure zero is countable.
\end{thm}
\begin{proof}
   Let $\mathcal{A}$ be a collection of sets $\{A_\alpha\}_{\alpha \in \Lambda}$ where $A_\alpha \subseteq \mathbb{R}$ and $\Lambda$ is arbitrary indexing set such that for each $\alpha \in \Lambda$, $A_\alpha$ is measurable, $\lambda(A_\alpha)>0$ and $\lambda(A_\alpha \cap A_\beta)=0$ whenever $\alpha \neq \beta$, then we are to show that $\mathcal{A}$ is countable. Let us assume that $\mathcal{A}$ is uncountable. Consider $\mathbb{R}$ as $\bigcup_{n \in \mathbb{Z}}[n,n+1]$ where $\mathbb{Z}$ is the set of integers. For any $\alpha \in \Lambda$, $A_\alpha = A_\alpha \cap \mathbb{R}= A_\alpha \cap (\bigcup_{k \in \mathbb{Z}}[k,k+1])=\bigcup_{k \in \mathbb{Z}}(A_\alpha \cap [k,k+1])$. Therefore, $\lambda(A_\alpha)=\sum_{k \in \mathbb{Z}} \lambda (A_\alpha \cap [k,k+1])$. Since $\lambda(A_\alpha)>0$ so there exists at least one $k$ such that  $\lambda (A_\alpha \cap [k,k+1])>0$. Thus for each $A_\alpha$ there exists some $i \in \mathbb{Z}$ such that $\lambda (A_\alpha \cap [i,i+1])>0$. Now if each interval $[k,k+1]$ for $k \in \mathbb{Z}$ intersect with only countably many $A_\alpha$'s such that $\lambda (A_\alpha \cap [k,k+1])>0$ then the collection $\{A_\alpha\}_{\alpha \in \Lambda}$ will be countable since countable union of countably many elements is again countable; which is a contradiction. Therefore, there exists some $k_0 \in \mathbb{Z}$ such that $[k_0,k_0+1]$ intersect with uncountably many $A_\alpha$ in $\mathcal{A}$ such that $\lambda (A_\alpha \cap [k_0,k_0+1])>0$. Take, $\Lambda^{'}=\{\alpha \in \Lambda: \lambda (A_\alpha \cap [k_0,k_0+1])>0\}$. Then clearly, $\Lambda^{'}$ is uncountable and $\Lambda^{'} \subseteq \Lambda$. Now $$\Lambda^{'}=\{\alpha \in \Lambda: \lambda (A_\alpha \cap [k_0,k_0+1])>0\}=\bigcup_{m \in \mathbb{N}} \{\alpha \in \Lambda: \lambda (A_\alpha \cap [k_0,k_0+1]) \geq \frac{1}{m}\}$$
   If each set in the above expression under union is countable then $\Lambda^{'}$ will be countable. So there exists some $m_0$ such that $\{\alpha \in \Lambda: \lambda (A_\alpha \cap [k_0,k_0+1]) \geq \frac{1}{m_0}\}$ is uncountable. 
   
   Let $\Lambda^{''}=\{\alpha \in \Lambda: \lambda (A_\alpha \cap [k_0,k_0+1]) \geq \frac{1}{m_0}\}$. Then $\Lambda^{''} \subseteq \Lambda^{'} \subseteq \Lambda$. Since we have assumed $\lambda(A_\alpha \cap A_\beta)=0$ whenever $\alpha \neq \beta$, so $$\lambda \left([k_0,k_0+1]\right) \geq \sum_{\alpha \in \Lambda^{'}} \lambda(A_{\alpha} \cap [k_0,k_0+1]) \geq \sum_{\alpha \in \Lambda^{''}} \lambda(A_{\alpha} \cap [k_0,k_0+1]) \geq \sum_{\alpha \in \Lambda^{''}} \frac{1}{m_0} = \infty$$
This is a contradiction. So, $\mathcal{A}$ must be a countable collection. This completes the proof.

   % Since there are uncountably many sets in $\mathcal{A}$ so one of the intervals $[i,i+1]$ has uncountably many sets $A_\alpha$ in $\mathcal{A}$ such that $\lambda (A_\alpha \cap [i,i+1])>0$ and name this indexing set as $\Lambda^{'} \subseteq \Lambda$. Let us take the sequence $\delta_k \rightarrow 0$. For some $\delta_k$ there must be uncountably many $A_\alpha$ for $\alpha \in \Lambda^{'}$ such that $\lambda(A_\alpha \cap [i,i+1])\geq \delta_k$ otherwise the cardinality of $\{A_\alpha\}_{\alpha \in \Lambda^{'}}$ would be countable. Let us take an indexing $\Lambda^{''} \subseteq \Lambda^{'}$ so that the collection of sets $\{A_\alpha\}_{\alpha \in \Lambda^{''}}$ is a subcollection of $\{A_\alpha\}_{\alpha \in \Lambda^{'}}$ such that each $A_{\alpha}$ for $\alpha \in \Lambda^{''}$ has measure greater than $\delta_k$. Since we have assumed that intersection of two distinct elements of $\mathcal{A}$ have measure zero so $$\lambda \left([i,i+1]\right) \geq \sum_{\alpha \in \Lambda^{'}} \lambda(A_{\alpha}) \geq \sum_{\alpha \in \Lambda^{''}} \lambda(A_{\alpha})$$
   % But $\sum_{\alpha \in \Lambda^{''}} \lambda(A_{\alpha}) = \infty > \lambda([i,i+1])$. This is a contradiction. So, $\mathcal{A}$ must be a countable collection. This completes the proof.
\end{proof}

Next we see under certain conditions, union of arbitrary collection of measurable sets can indeed become measurable.

\begin{thm}
    If $\{A_\alpha\}_{\alpha \in \Lambda}$ is an arbitrary collection of measurable sets where $\Lambda$ is arbitrary indexing set, such that for all $\alpha \in \Lambda$, $A_{\alpha} \subseteq \Theta_{\mathcal{I}}(A_{\alpha})$ and $\lambda(A_\alpha \setminus B)=0$ for any measurable set $B$ so that $B \subseteq \bigcup_{\alpha \in \Lambda}A_{\alpha}$ then $\bigcup_{\alpha \in \Lambda}A_{\alpha}$ is measurable.
\end{thm}
\begin{proof}
    Since for all $\alpha \in \Lambda$, $\lambda(A_\alpha \setminus B)=0$ thus by lemma \ref{8}, $\Theta_{\mathcal{I}}(A_\alpha) \subseteq \Theta_{\mathcal{I}}(B)$. So, $\bigcup_{\alpha \in \Lambda}\Theta_{\mathcal{I}}(A_{\alpha}) \subseteq \Theta_{\mathcal{I}}(B)$. Since, $B$ is measurable so by Theorem \ref{4}, $\lambda(B \triangle \Theta_\mathcal{I}(B))=0$ which implies $\Theta_\mathcal{I}(B)$ is measurable. Thus, $$B \subseteq \bigcup_{\alpha \in \Lambda}A_{\alpha} \subseteq \bigcup_{\alpha \in \Lambda}\Theta_{\mathcal{I}}(A_{\alpha}) \subseteq \Theta_{\mathcal{I}}(B)$$
    Since, $B$ and $\Theta_{\mathcal{I}}(B)$ are both measurable and they differ by a null set, so $\bigcup_{\alpha \in \Lambda}A_{\alpha}$ is measurable.
\end{proof}

	\begin{thm}
		The collection $\mathcal{T}^{\mathcal{I}}$ is a topology on $\mathbb{R}$.
	\end{thm} 
	\begin{proof}
 Clearly by lemma \ref{9}, $\Theta_{\mathcal{I}}(\emptyset)=\emptyset$ and $\Theta_{\mathcal{I}}(\mathbb{R})=\mathbb{R}$ and both $\emptyset$ and $\mathbb{R}$ are measurable. So, $\emptyset, \mathbb{R} \in \mathcal{T}^{\mathcal{I}}$.  Now we are to show $\mathcal{T}^{\mathcal{I}}$ is closed under finite intersection. Let $\{A_{\alpha_1}, A_{\alpha_2}, \cdots, A_{\alpha_n}\}$ be any finite collection in $\mathcal{T}^{\mathcal{I}}$. So each $A_{\alpha_k}$ is measurable and $A_{\alpha_k} \subseteq \Theta_{\mathcal{I}}(A_{\alpha_k})$ for each $k$. Clearly, $\bigcap_{k=1}^{n}A_{\alpha_k}$ is measurable. By Theorem \ref{7}, $$\bigcap_{k=1}^{n}A_{\alpha_k} \subseteq \bigcap_{k=1}^{n} \Theta_{\mathcal{I}}(A_{\alpha_k})=\Theta_{\mathcal{I}}(\bigcap_{k=1}^{n}A_{\alpha_k})$$
 Therefore, $\bigcap_{k=1}^{n}A_{\alpha_k} \in \mathcal{T}^{\mathcal{I}}$.

 Next we are to show $\mathcal{T}^{\mathcal{I}}$ is closed under arbitrary unions. If $\{A_\alpha\}_{\alpha \in \Lambda}$ is an arbitrary collection of sets in $\mathcal{T}^{\mathcal{I}}$ where $\Lambda$ is arbitrary indexing set then $\bigcup_{\alpha \in \Lambda}A_{\alpha} \in \mathcal{T}^{\mathcal{I}}$ i.e. we are to show $\bigcup_{\alpha \in \Lambda}A_{\alpha} \subseteq \Theta_{\mathcal{I}}(\bigcup_{\alpha \in \Lambda}A_{\alpha})$ and $\bigcup_{\alpha \in \Lambda}A_{\alpha}$ is measurable.

 Since for each  $\alpha \in \Lambda$, $A_{\alpha} \in \mathcal{T}^{\mathcal{I}}$ we have $A_{\alpha} \subseteq \Theta_{\mathcal{I}}(A_{\alpha})$ and it follows that $\bigcup_{\alpha \in \Lambda}A_{\alpha} \subseteq \bigcup_{\alpha \in \Lambda}\Theta_{\mathcal{I}}(A_{\alpha})$. Let $x \in \bigcup_{\alpha \in \Lambda}\Theta_{\mathcal{I}}(A_{\alpha})$ then there exists $\beta \in \Lambda$ such that $x \in \Theta_{\mathcal{I}}(A_{\beta})$. Note that $A_{\beta} \subset \bigcup_{\alpha \in \Lambda}A_{\alpha}$ so $\Theta_{\mathcal{I}}(A_{\beta}) \subseteq \Theta_{\mathcal{I}}(\bigcup_{\alpha \in \Lambda}A_{\alpha})$. Thus $x \in \Theta_{\mathcal{I}}(\bigcup_{\alpha \in \Lambda}A_{\alpha})$. So, $\bigcup_{\alpha \in \Lambda}\Theta_{\mathcal{I}}(A_{\alpha}) \subseteq \Theta_{\mathcal{I}}(\bigcup_{\alpha \in \Lambda}A_{\alpha})$. Therefore, $\bigcup_{\alpha \in \Lambda}A_{\alpha} \subseteq \Theta_{\mathcal{I}}(\bigcup_{\alpha \in \Lambda}A_{\alpha})$.

 It remains to show that arbitrary union of members of $\mathcal{T}^{\mathcal{I}}$ is measurable. Let $\{A_\alpha\}_{\alpha \in \Lambda}$ be an arbitrary collection of sets in $\mathcal{T}^{\mathcal{I}}$ where $\Lambda$ is arbitrary indexing set. Since by lemma \ref{9}, $\lambda(A)=0$ implies $\Theta_{\mathcal{I}}(A)=\emptyset$ which in turn implies $A \notin \mathcal{T}^{\mathcal{I}}$ so clearly, $\lambda(A_\alpha)>0$ for all $\alpha$. We choose a sequence in $\Lambda$ in the following way. By Well Ordering Principle, every set can be well ordered. So we can linearly order the elements of $\Lambda$. Choose the first element of $\Lambda$ to be $\alpha_0$. Following the linear order on $\Lambda$ compare each element $A_{\alpha^{'}}$ with $A_{\alpha_0}$. If $\lambda(A_{\alpha^{'}} \setminus A_{\alpha_0})>0$ let $\alpha_1 = \alpha^{'}$. If no such $\alpha^{'}$ exists let us take the sequence to be $(\alpha_0, \alpha_0, \cdots)$. Once $\alpha_1$ is chosen search through $\Lambda$ starting after $\alpha_1$ to find $A_{\alpha_2}$ such that $\lambda(A_{\alpha_2} \setminus (A_{\alpha_0} \cup A_{\alpha_1}))>0$. If no such $\alpha_2$ exists then take the sequence as $(\alpha_0, \alpha_1, \alpha_1, \cdots)$.  Continuing for each $n \in \mathbb{N}$ at any step $m$, assuming $\alpha_{m-1}$ is already chosen, search through $\Lambda$ starting after $\alpha_{m-1}$ to find $A_{\alpha_m}$ such that $\lambda(A_{\alpha_m} \setminus \bigcup_{n=0}^{m-1}A_{\alpha_n})>0$. If no $A_{\alpha_m}$ can be found let the sequence be $(\alpha_0, \alpha_1, \cdots, \alpha_{m-1}, \alpha_{m-1}, \cdots)$. Whether or not a unique $\alpha_{n}$ can be found for each $n$, by Theorem $4.1$ the sequence may be atmost countably long. So we obtain a sequence $\{\alpha_n\}_{n \in \mathbb{N}}$ such that for any $\alpha \in \Lambda$ we have $\lambda(A_{\alpha} \setminus \bigcup_{n=0}^{\infty}A_{\alpha_n})=0$. Since $\{A_{\alpha_n}\}$ is a countable sequence of measurable sets so $\bigcup_{n=0}^{\infty} A_{\alpha_n}$ is measurable. Now by Lemma \ref{8} for any $\alpha \in \Lambda$, 
 $$\lambda(A_{\alpha} \setminus \bigcup_{n=0}^{\infty}A_{\alpha_n})=0 \implies \Theta_{\mathcal{I}}(A_{\alpha}) \subseteq \Theta_{\mathcal{I}}(\bigcup_{n=0}^{\infty}A_{\alpha_n})$$
 Hence, $\bigcup_{\alpha \in \Lambda}\Theta_{\mathcal{I}}(A_{\alpha}) \subseteq \Theta_{\mathcal{I}}\left(\bigcup_{n=0}^{\infty}A_{\alpha_n}\right)$.
 By Lebesgue $\mathcal{I}$-density theorem \ref{4}, $$\lambda \left(\bigcup_{n=0}^{\infty}A_{\alpha_n} \triangle \Theta_{\mathcal{I}}\left(\bigcup_{n=0}^{\infty}A_{\alpha_n}\right)\right)=0$$
 So, $\Theta_{\mathcal{I}}\left(\bigcup_{n=0}^{\infty}A_{\alpha_n}\right)$ is measurable. Thus, 
 $$\bigcup_{n=0}^{\infty}A_{\alpha_n} \subseteq \bigcup_{\alpha \in \Lambda}A_{\alpha} \subseteq \bigcup_{\alpha \in \Lambda}\Theta_{\mathcal{I}}(A_{\alpha}) \subseteq \Theta_{\mathcal{I}}\left(\bigcup_{n=0}^{\infty}A_{\alpha_n}\right)$$
 Since, $\bigcup_{n=0}^{\infty}A_{\alpha_n}$ and $\Theta_{\mathcal{I}}\left(\bigcup_{n=0}^{\infty}A_{\alpha_n}\right)$ both are measurable and differ by a null set, so $\bigcup_{\alpha \in \Lambda}A_{\alpha}$ is measurable. Thus, $\mathcal{T}^{\mathcal{I}}$ is closed under arbitrary unions. This completes the proof.
 
	\end{proof}
	
	We name the topology $\mathcal{T}^{\mathcal{I}}$ to be the $\mathcal{I}$-density topology on $\mathbb{R}$ and the pair $(\mathbb{R},\mathcal{T}^{\mathcal{I}})$ is the corresponding topological space. 

 \begin{dfn}(\cite{Oxtoby})
     If a class of subsets of an arbitrary set X is closed under countable
unions, complementations, and contains the set X itself, then the class is called a $\sigma$-algebra.
 \end{dfn}

The $\sigma$-algebra that is generated by the open sets of any given topology is the collection of Borel sets of that topology.

\begin{dfn}(\cite{Oxtoby})
    Let $\mathcal{T}$ be any topology. The collection of Borel sets of
$\mathcal{T}$ is the smallest $\sigma$-algebra containing the open sets of $\mathcal{T}$.
\end{dfn}

The collection of Borel sets can be characterized as the $\sigma$-algebra generated by the open sets. Thus, we can talk about Borel sets on any given topological space. The $\mathcal{I}$-density topology requires open sets to be Lebesgue measurable but not all measurable sets are open in $(\mathbb{R},\mathcal{T}^{\mathcal{I}})$. In \cite{Scheinberg} S. Scheinberg proved the existence of a topology on the space of reals in which the Borel sets are precisely the Lebesgue measurable sets. Likewise we give a characterization of Lebesgue measurable subsets of reals in the following theorem.

 \begin{thm}
   The Borel sets in the $\mathcal{I}$-density topology $\mathcal{T}^{\mathcal{I}}$ on the space of reals are precisely the Lebesgue measurable sets.
 \end{thm}

 \begin{proof}
     Let $B$ be a Borel set in $\mathcal{T}^{\mathcal{I}}$. So $B$ is formed through the operation of countable union, countable intersection and relative complement of sets in $\mathcal{T}^{\mathcal{I}}$. Since each element of $\mathcal{T}^{\mathcal{I}}$ is measurable so $B$ is measurable. Conversely, let $B$ be lebesgue measurable. Then by theorem \ref{4} we can write $B=C \cup D$ where 
 $C=B \cap \Theta_{\mathcal{I}}(B)$ and $D$ is measure zero set. Then clearly the set $C$ is measurable since both $B$ and $\Theta_{\mathcal{I}}(B)$ are measurable. Next, $\Theta_{\mathcal{I}}(C)= \Theta_{\mathcal{I}}(B) \cap \Theta_{\mathcal{I}}(\Theta_{\mathcal{I}}(B))=\Theta_{\mathcal{I}}(B) \supseteq C$. Therefore, the set $C$ is $\mathcal{T}^{\mathcal{I}}$-open. Again, $\lambda(D)=0$, so $D$ is measurable which implies $\mathbb{R} \setminus D$ is measurable and  by lemma \ref{9}, $\Theta_{\mathcal{I}}(\mathbb{R} \setminus D)=\mathbb{R} \supseteq \mathbb{R}\setminus D$. So, $\mathbb{R} \setminus D$ is $\mathcal{T}^{\mathcal{I}}$-open and consequently $D$ is $\mathcal{T}^{\mathcal{I}}$-closed. Thus, $B$ is union of $\mathcal{T}^{\mathcal{I}}$-open and $\mathcal{T}^{\mathcal{I}}$-closed set. So, $B$ is a Borel set in $\mathcal{T}^{\mathcal{I}}$. Therefore, the result follows.
 \end{proof}

	% \begin{thm}
	% 	The family $\mathfrak{T}_{\mathcal{I}}$ is a topology on the real line finer than the density topology $\mathfrak{T}_d$.
	% \end{thm}
	% \begin{proof} To write
	% \end{proof}
	
	% \begin{rmrk}
	% 	It is shown in \cite{Density topologies} $\mathfrak{T}_d$ is finer than the natural topology on the real line. The inclusion $\mathfrak{T}_d \subset \mathfrak{T}_{\mathcal{I}}$ implies $\mathfrak{T}_{\mathcal{I}}$ is finer than the natural topology on the real line. However we give an alternative proof that $\mathcal{I}$-density topology is finer than usual topology on $\mathbb{R}$.
	% \end{rmrk}
	
	% \begin{thm}
	% 	The family $\mathfrak{T}_{\mathcal{I}}$ is a topology on the real line finer than the natural topology $\mathfrak{T}_U$.
	% \end{thm}
	% \begin{proof}
		
	% \end{proof}

	% \begin{dfn}
	% 	A set $F \subset \mathbb{R}$ is said to be $\mathcal{I}-d$ closed if $F^c$ is $\mathcal{I}-d$ open.
	% \end{dfn}
	
	% \begin{dfn}
	% 	A point $x \in \mathbb{R}$ is a $\mathcal{I}-d$ limit point of a set $E \subset \mathbb{R}$ (not necessarily measurable) if and only if $\mathcal{I}-d^{-}(x,E)>0$ where instead of taking measure $\lambda$ outer measure $\lambda^{*}$ is taken.
	% \end{dfn}
	
	% \begin{thm} In the space $(\mathbb{R}, \mathfrak{T}_\mathcal{I})$ given any Lebesgue measurable set $E \subset \mathbb{R}$, $\lambda(E)=0$ if and only if $E$ is $\mathcal{I}-d$ closed and discrete.
	% \end{thm}
	% \begin{proof}
	% \end{proof}

	\section*{Acknowledgements}
	\textit{The first author is thankful to The Council of Scientific and Industrial Research (CSIR), Government of India, for giving the award of Senior Research Fellowship (File no. 09/025(0277)/2019-EMR-I) during the tenure of preparation of this research paper.}

\end{document}